\documentclass{amsart}
\usepackage{amscd,euscript,color}
\usepackage{amsfonts}
\usepackage{theoremref}
\usepackage{amsmath}
\usepackage{amssymb}
\usepackage{amsthm}
\usepackage{MnSymbol}
\usepackage{lipsum}
\usepackage{amsfonts}
\usepackage{tikz-cd}
\usepackage{bm}
\usepackage{tikz-cd}
\usepackage{dsfont}
\usepackage{mathtools}
\usepackage[colorlinks]{hyperref}
\usepackage{color}
\usepackage[alphabetic,initials,nobysame]{amsrefs}
\usepackage{graphicx}
\usepackage{enumerate}
\usepackage{enumitem}
\usepackage{comment}
\usepackage{color}
\hypersetup{
colorlinks=false,
linkcolor=black,
filecolor=black,      
urlcolor=black,
}
\newcommand{\Addresses}{{
  \bigskip
  \footnotesize

  Rahul Singh, \textsc{Department of Mathematics, Louisiana State University, Baton Rouge, LA, USA}\par\nopagebreak
  \textit{E-mail address}: \texttt{rahulsingh@lsu.edu}

  \medskip

  Anton M. Zeitlin, \textsc{School of Mathematics, Georgia Institute of Technology, Atlanta, GA, USA}\par\nopagebreak
  \textit{E-mail address}: \texttt{zeitlin@gatech.edu}
}}

\theoremstyle{plain}
\newtheorem{theorem}{Theorem}[section]
\newtheorem{proposition}{Proposition}[section]
\newtheorem{corollary}{Corollary}[section]
\newtheorem{lemma}[theorem]{Lemma}
\newtheorem{fact}{Fact}[section]

\newtheorem{notation}{Notation}[section]

\theoremstyle{definition}
\newtheorem{definition}{Definition}

\theoremstyle{remark}
\newtheorem{remark}{Remark}[section]
\newtheorem{example}{Example}
\newtheorem{algorithm}{Algorithm}
\title{\MakeLowercase{$QQ$-systems and tropical geometry}}
\author{Rahul Singh, Anton M. Zeitlin}
\date{}
\DeclareMathOperator{\Def}{\text{Def}}
\DeclareMathOperator{\trop}{\text{trop}}
\DeclareMathOperator{\rank}{\text{Rank}}

\makeatletter
\def\subsubsection{\@startsection{subsubsection}{3}%
\z@{.5\linespacing\@plus.2\linespacing}{-.5em}%
{\smallfont\bfseries}}
\makeatother

\begin{document}
\begin{abstract}
We investigate the system of polynomial equations, known as $QQ$-systems, which are closely related to the so-called Bethe ansatz equations of the XXZ spin chain, using the methods of tropical geometry. 
\end{abstract}
\maketitle

\tableofcontents
\section{Introduction}

The $QQ$-system is a system of difference equations that emerge in various representation-theoretic contexts. In the theory of quantum integrable systems, they appear as the relations between eigenvalues of the renowned {\it Baxter operators} \cite{Baxterbook}. 
The underlying representation theory of affine quantum groups provides a theoretical basis \cite{BLZ3}, \cite{FH}, \cite{FH2}, \cite{HJ} behind the construction of these operators, where the $QQ$-system describes the relations between generators of the extended Grothendieck ring of finite-dimensional representations of an affine quantum group, and Baxter operators are the twisted half-traces of certain R-matrix operators acting within this extended ring. 
Here, the Cartan-valued ``twist" parameter corresponds to the twisted boundary conditions for integrable models. 
In this context, the eigenvalues of Baxter operators are polynomials of one variable, and the roots of these polynomials can be identified upon certain non-degeneracy conditions with the solutions of the so-called {\it Bethe equations} (see, e.g., \cite{KBI}, \cite{Faddeev}, \cite{R}), characterizing the spectrum of the related integrable model, known as the XXZ model. We note that the QQ-systems emerged naturally in the context of the study of Bethe equations as well (see, e.g., \cite{Krichever},\cite{MVdiscrete},  \cite{MVpopulations}, \cite{MVquasipolynomials}).

One can consider various limits of the $QQ$-systems, corresponding to Bethe equations associated with quantum integrable systems based on Yangians (XXX models) as well as just simple Lie algebras, known as Gaudin models. However, these limits lack such an explicit representation-theoretic presentation in terms of an extended Grothendieck ring of representations.

At the same time, the Gaudin integrable models and the related limit of the $QQ$-system, the so-called $qq$-system, which is a system of differential equations, have a geometric interpretation from a completely different angle: from the point of view of geometric Langlands correspondence. In this case, the $qq$-system characterizes the connections of a specific type on a projective line, called opers, corresponding to the group with Langlands dual Lie algebra. In this case, the twist parameters define the constant Cartan connection, which these opers are gauge equivalent to.  
This relation between \cite{FFR_Opers}, \cite{F_Opers}, \cite{F_Gaudin} Gaudin integrable models characterizing the $D$-modules on the moduli stack of $G$-bundles and local systems for the Langlands dual group on the projective line provided one of the simplest nontrivial examples of geometric Langlands correspondence \cite{F_Loop}.

Recently, this example was successfully deformed \cite{FKSZ}, \cite{KSZ}, \cite{KZ2}, \cite{KZ3} to incorporate various possible $QQ$-systems, which use the multiplicative version of the connection on the projective line, known as the $q$-oper (see also earlier work \cite{MVdiscrete}). The twist parameters are incorporated in the Cartan-valued connections, which these multiplicative connections are $q$-gauge equivalent to. From this point of view, the $QQ$-systems can be interpreted as the relation between generalized minors for certain associated meromorphic sections of principal bundles, generalizing older determinantal formulas related to the so-called Lewis Carroll formulas, which were used before in the theory of Bethe ansatz equations in the $\mathfrak{sl}_n$-case.

Simultaneously, the recently discovered connection proposed by Nekrasov and Shatashvili in the framework of theoretical physics \cite{NS1}, \cite{NS2} between integrable models based on quantum groups / Yangians \cite{CP} and the enumerative geometry of Nakajima quiver varieties \cite{G}, and continued with a string of mathematical results, e.g.,  \cite{BMO} \cite{MO}, \cite{OS}, \cite{O}, \cite{PSZ}, \cite{KPSZ}, \cite{KSZ}, \cite{Zlectures}, \cite{Zaslow}, gave a new interpretation of the $QQ$-system, as the relations within the equivariant quantum K-theory / cohomology ring of the corresponding variety.  The K\"ahler parameters that provide the deformation of the quantum K-theory/cohomology ring give another realization of the twist parameters for the $QQ$-systems. When K\"ahler parameters become zero, the $QQ$-system describes the relations within the classical equivariant cohomology / K-theory, which is a much simpler system of polynomial equations, which are much easier to solve.

In this paper, we call this ``classical" limit an {\it infinite} $QQ$-system\footnote{This name is because it is more natural for our study to take as parameters the inverse to the K\"ahler ones.}. We try to answer the natural question: How can one construct explicit solutions of the $QQ$-systems around such infinite solutions? 
One can partially read the answer to this question for $QQ$-systems corresponding to integrable models based on quantum groups and using the interpretation of polynomial solutions as eigenvalues of Baxter operators. That can give insight into the existence of such deformations, but it does not provide an explicit answer. Moreover, for $QQ$-systems and $qq$-systems corresponding to integrable models based on Yangians and Gaudin models, respectively, one has to rely on the appropriate limits of the deformation parameter.   

Instead, in this paper, we use the methods of tropical geometry that extend the classes of $qq/QQ$-systems for which one can construct analytic solutions in terms of twist parameters. 
For an arbitrary simple Lie algebra $\mathfrak{g}$ (resp. connected, simply connected, simple algebraic group $G$) of rank $r$, we prove that such analytic solutions exist for any isolated solution of the infinite $qq$ (resp. $QQ$)-system. This statement allows us to directly prove some structural theorems for the corresponding opers and their deformations. 
Moreover, in the simplest non-trivial case of $qq/QQ$-systems related to the $\mathfrak{sl}_2$ Lie algebra, we provide an algorithmic way of writing down the solutions for a large enough value of a deformation parameter.

We note here also some parallel recent progress, e.g., \cite{Volin}, \cite{nepomechie}, within the theoretical physics community in finding solutions to  Bethe ansatz equations via $QQ$-systems, using a combination of representation-theoretic realizations and algorithms. 

Let us give an outline of the structure of the article. In Section \ref{preliminaries}, we recall the definitions of $qq$-systems and $QQ$-systems. 
In Section \ref{Motivation} we give motivation for the $QQ$-systems, their classical limit, and the way they emerge in the enumerative geometry of quiver varieties and the theory of integrable systems. In Section \ref{mainresults}, we state the main results of this article. In Section \ref{tropicalgeometry}, we recall a few definitions and standard facts from tropical algebraic geometry. In Sections \ref{proofqq} and \ref{proofQQ}, we give proofs of deformation for $qq$-systems and $QQ$-systems respectively. We give an application of deformation in the $\mathfrak{sl}_2$-case and $SL_{2}$-case in Sections \ref{applicationsl2} and \ref{applicationspinchain}, respectively. 

\subsection{Acknowledgements} The first named author is grateful to Josephine Yu and Kiumars Kaveh for answering various questions related to tropical geometry. The first named author would also like to thank Nikkos Svoboda for helping him setup Gfan. The second named author is partially supported by the NSF grant DMS-2526435 (formerly DMS-2203823).

\section{Preliminaries}\label{preliminaries}
Let us introduce the following notations.
\begin{notation}
By $G$ we will denote a connected, simply connected, simple algebraic group of rank $r$ over $\mathbb{C}$. Let $H$ be a maximal torus of $G$ and let $B_{-}$ be a Borel subgroup of $G$ containing $H$. Let $N_{-}$ be the unipotent radical of $B_{-}$ and let $B_{+}$ be the opposite Borel subgroup containing $H$. Let $\{\alpha_{1},\ldots,\alpha_{r}\}$ be the set of simple roots corresponding to the choice of the pair $(B_{+},H)$. Let $\{ \check{\mathrm{\alpha}}_{1},\ldots,\check{\mathrm{\alpha}}_{r}\}$ be the corresponding coroots and let $(a_{ij})$ denote the associated Cartan matrix (recall that $a_{ij}=\langle\alpha_{j},\check{\mathrm{\alpha}}_{i}\rangle$). Let $\mathfrak{g}$, $\mathfrak{h}$, $\mathfrak{b}_{-}$, $\mathfrak{b}_{+}$ and $\mathfrak{n}_{-}$ denote the Lie algebras of $G$, $H$, $B_{-}$, $B_{+}$ and $N_{-}$ respectively. We denote by $\{e_{i},f_{i},\check{\mathrm{\alpha}}_{i}\}_{i=1,\ldots,r}$ the corresponding Chevalley generators.
\end{notation}
\subsection{$qq$-systems}
A $qq$-system is a non-linear system of differential equations depending on $r$ polynomials $\Lambda_{1}(z),\ldots,\Lambda_{r}(z)$ and a semisimple element $Z^{H}\in\mathfrak{h}$. More precisely (see, e.g., \cite{MVquasipolynomials}, \cite{BSZ}, \cite{Zwronskians}) we have the following definition:
\begin{definition}\label{finiteqq}
The $qq$-system associated to $\mathfrak{g}$, monic polynomials $\Lambda_{1}(z),\ldots,\Lambda_{r}(z)$ and a semisimple element $Z^{H}\in\mathfrak{h}$ is the system of equations
\begin{equation}\label{qqsystem}
W(q_{+}^{i},q_{-}^{i})(z)
+
\langle\alpha_{i},Z^H\rangle
q_{+}^{i}(z)q_{-}^{i}(z)
=
\Lambda_{i}(z)\prod_{j\neq i}\big[q_{+}^{j}(z)\big]^{-a_{ji}},\quad i=1,\ldots,r,
\end{equation}
where the Wronskian $W(f,g)(z)$ of two rational functions $f(z)$ and $g(z)$ is given by
\[
W(f,g)(z)
=
f(z)\partial_{z}g(z)-g(z)\partial_{z}f(z).
\]
\end{definition}
\begin{definition}
    The infinite $qq$-system associated to $\mathfrak{g}$ and monic polynomials $\Lambda_{1}(z),\ldots,\Lambda_{r}(z)$ is the system of equations
\begin{equation}\label{infiniteqqsystem}
q_{+}^{i}(z)q_{-}^{i}(z)
=
\Lambda_{i}(z)\prod_{j\neq i}\big[q_{+}^{j}(z)\big]^{-a_{ji}},\quad i=1,\ldots,r,
\end{equation}
where the $q_{+}^{j}$'s are assumed to be monic.
\end{definition}
\begin{remark}\label{qqfinitetoinfinite}
    The system \eqref{infiniteqqsystem} is obtained from the system \eqref{qqsystem} by setting $\tilde{q}_{-}^{i}(z):=\xi_{i} q_{-}^{i}(z)$, where $\xi_{i}:=\langle\alpha_{i},Z^H\rangle$ and letting $\xi_{i}\to\infty$ for each $i$.
\end{remark}
\begin{example}
   For $\mathfrak{g}=\mathfrak{sl}_{2}$ and $Z^{H}\neq 0$, the $qq$-system is the equation 
   \begin{equation}\label{sl2qq}
       q_{+}(z)q_{-}(z)+ tW(q_{+},q_{-})(z)
=
\Lambda(z)
   \end{equation}
where $t:=\xi^{-1}$ and the infinite $qq$-system is the equation 
\begin{equation}\label{sl2infiniteqq}
    q_{+}(z)q_{-}(z)
=
\Lambda(z)
\end{equation} for a monic polynomial $\Lambda(z)$.
\end{example}
We now refer to \cite{BSZ} for definitions and results, where it is shown that for a simply-laced $\mathfrak{g}$ and $Z^{H}$ regular semisimple, every nondegenerate $Z^{H}$-twisted Miura-Pl\"{u}cker oper with admissible combinatorics is a nondegenerate $Z^{H}$-twisted Miura oper under the assumption that $\Lambda_{i}(z)$'s are separable. This in turn gives a one-to-one correspondence between nondegenerate $Z^H$-twisted Miura opers and solutions of the Bethe Ansatz equations \eqref{betheqqg}. In their proof, the authors use the existence of deformations of solutions of given infinite $qq$-system in order to construct Backl\"{u}nd transforms of a certain initial solution of a finite $qq$-system (see \cite{BSZ}*{Section 6}). In this paper, we prove the existence of deformations of such solutions with no restrictions on $\Lambda_{i}(z)$'s for an arbitrary simple Lie algebra $\mathfrak{g}$.
\subsection{$QQ$-systems}
A $QQ$-system is a non-linear system of difference equations depending on $r$ polynomials $\Lambda_{1}(z),\ldots,$ $\Lambda_{r}(z)$ and an element $Z\in H$. More precisely, we fix $\Lambda_{i}(z)$'s and may assume without loss of generality that $\Lambda_{i}(z)$'s are monic and let $\{\zeta_{i}\}_{i=1,\ldots,r}$ be the non-zero complex numbers that correspond to $Z\in H$ via the isomorphism (recall that $G$ is assumed to be simply connected):
\[
(\mathbb{C}^{\times})^r\xrightarrow{\simeq}H, \quad (c_{1},\ldots,c_{r})\mapsto\prod_{i}\check{\mathrm{\alpha}}_{i}(c_{i}).
\]
In addition, we assume that $Z$ satisfies the following property:
\[
\prod_{i}\zeta_{i}^{a_{ij}}\not\in q^{\mathbb{Z}}, \quad 1\leq j\leq r.
\]
The above condition implies that $Z$ is, in particular, regular semisimple. Here is an explicit definition (see, e.g., \cite{MVquasipolynomials},\cite{FH2},\cite{FKSZ}):
\begin{definition}
The $QQ$-system associated to $G$, monic polynomials $\Lambda_{1}(z),\ldots,\Lambda_{r}(z)$ and $Z\in H$ is the system of equations
\begin{equation}\label{QQsystem}
\tilde{\xi_{i}}Q_{+}^{i}(qz)Q_{-}^{i}(z)-\xi_{i} Q_{+}^{i}(z)Q_{-}^{i}(qz)
=
\Lambda_{i}(z)\prod_{j\neq i}\big[Q_{+}^{j}(z)\big]^{-a_{ji}},\quad i=1,\ldots,r,
\end{equation}
where the $Q_{+}^{j}$'s are assumed to be monic and $\tilde{\xi_{i}}$ and $\xi_{i}$ are defined as:
\[
\tilde{\xi_{i}}=\zeta_{i}\prod_{j>i}\zeta_{j}^{a_{ji}}, 
\quad
\xi_{i}=\zeta_{i}^{-1}\prod_{j<i}\zeta_{j}^{-a_{ji}}.
\]
\end{definition}
\begin{definition}\label{infiniteQQ}
    The infinite $QQ$-system associated to $G$ and monic polynomials $\Lambda_{1}(z),\ldots,\Lambda_{r}(z)$ is the system of equations
\begin{equation}\label{infiniteQQsystem}
Q_{+}^{i}(qz)Q_{-}^{i}(z)
=
\Lambda_{i}(z)\prod_{j\neq i}\big[Q_{+}^{j}(z)\big]^{-a_{ji}},\quad i=1,\ldots,r,
\end{equation}
where the $Q_{+}^{j}$'s are assumed to be monic.
\end{definition}
\begin{remark}\label{QQfinitetoinfinite}
    The system \eqref{infiniteQQsystem} is obtained from the system \eqref{QQsystem} by setting $\tilde{Q}_{-}^{i}:=\tilde{\xi_{i}} Q_{-}^{i}$, and letting $\hat{\xi}_{i}\to\infty$ for each $i$, where $\hat{\xi_{i}}:=\prod_{j}\zeta_{j}^{a_{ji}}$. 
\end{remark}
\begin{example}
   For $G=SL_{2}$, the $QQ$-system is the equation
   \begin{equation}\label{sl2QQ}
       \zeta Q_{+}(qz)Q_{-}(z)-\zeta^{-1} Q_{+}(z)Q_{-}(qz)
=
\Lambda(z)
\end{equation}
and the infinite $QQ$-system is the equation 
\begin{equation}\label{infinitesl2QQ}    
   Q_{+}(qz)Q_{-}(z)
=
\Lambda(z)
\end{equation}
obtained by letting $t\to 0$, where $t:=\zeta^{-2}$.
\end{example}
\subsection{Puiseux series}\label{basicsonpuiseuxseries}
Recall that a Puiseux series with coefficients in $\mathbb{C}$ is an expression of the form
    \[
    \sum_{k=k_{0}}^{\infty}c_{k}t^{k/n}
    \]
    where $n\in\mathbb{Z}_{>0}$, $k_{0}\in\mathbb{Z}$ and $c_{k}\in\mathbb{C}$.
The Puiseux series with coefficients in $\mathbb{C}$ form a field, which we will denote by $L$. The natural valuation $\nu:L\backslash \{0\}\rightarrow \mathbb{Q}$, $f\mapsto k_{0}/n$ makes $L$ into a valued field. We define the valuation of $0\in L$ as $+\infty$.
\begin{definition}\label{convergentpuiseuxseries}
    A Puiseux series $f\in L$ is said to be convergent if there exists a neighborhood $U$ of $0$ such that $f$ is convergent on $U$ (resp. $U\backslash\{0\}$) when $\nu(f)\geq 0$ (resp. $\nu(f)<0$).
\end{definition}
The convergent Puiseux series with coefficients in $\mathbb{C}$ form a field, which we will denote by $L_{c}$. The restriction of the valuation $\nu$ to $L_{c}\backslash \{0\}$ makes $L_{c}$ into a valued field with the residue field $\mathbb{C}$.
\begin{fact}
    The field $L_{c}$ is algebraically closed.
\end{fact}

\section{Motivation: Quiver Varieties, \texorpdfstring{$QQ$}{QQ}-Systems, and Integrable Models}
\label{Motivation}
\subsection{Quiver Varieties and Infinite $QQ$-Systems}

The $QQ$-systems emerge naturally in the enumerative geometry of Nakajima quiver varieties, which serve as fundamental examples of symplectic resolutions (see, e.g., \cite{G}, \cite{Kaledinreview}, \cite{Kamnitzer:2022aa}).

A \textit{quiver} is a directed graph defined by a set of vertices $I$ and oriented edges. A \textit{framed quiver} extends this by doubling the vertex set: for each original vertex, a new \textit{framing vertex} is added, connected by an edge directed from the framing vertex to its original counterpart.

For a framed quiver, we associate:
\begin{itemize}
    \item Vector spaces $V_i$ to original vertices and $W_i$ (termed \textit{flavor spaces} in physics) to framing vertices, with dimensions $\mathbf{v}_i = \dim V_i$ and $\mathbf{w}_i = \dim W_i$. Morphisms between these spaces correspond to quiver edges, with the incidence matrix $Q_{ij}$ counting oriented edges from vertex $i$ to $j$.
    \item An affine representation space $M = \text{Rep}(\mathbf{v}, \mathbf{w})$, defined as:
    \[
    M = \bigoplus_{i \in I} \text{Hom}(W_i, V_i) \oplus \bigoplus_{i,j \in I} Q_{ij} \otimes \text{Hom}(V_i, V_j).
    \]
\end{itemize}

The space $M$ admits a natural action of the group $G_{\mathbf{v}} = \prod_{i \in I} GL(V_i)$. In gauge theory contexts requiring sufficient supersymmetry, we consider the cotangent space $T^*M$, where $G_{\mathbf{v}}$ acts via a Hamiltonian action with moment map $\mu: T^*M \to \mathfrak{g}_{\mathbf{v}}^*$. Define $L_{\mathbf{v}, \mathbf{w}} = \mu^{-1}(0)$, then \textit{Nakajima quiver variety} is then constructed as the symplectic reduction:
\[
X = N_{\mathbf{v}, \mathbf{w}} = L_{\mathbf{v}, \mathbf{w}} /\!\!/_{\theta} G_{\mathbf{v}} = L^{ss}_{\mathbf{v}, \mathbf{w}} / G_{\mathbf{v}},
\]
where $L^{ss}_{\mathbf{v}, \mathbf{w}}$ denotes the semi-stable locus, determined by a stability parameter $\theta \in \mathbb{Z}^I$ (see, e.g, \cite{G} for details).

Consider a type $A_n$ quiver as an example:
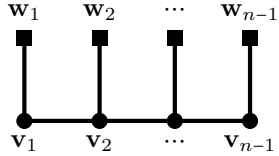
\begin{center}
\begin{tikzpicture}
    \draw [ultra thick] (0,0) -- (3,0);
    \draw [fill] (0,0) circle [radius=0.1];
    \draw [fill] (1,0) circle [radius=0.1];
    \draw [fill] (2,0) circle [radius=0.1];
    \draw [fill] (3,0) circle [radius=0.1];
    \node at (0,-0.3) {$\mathbf{v}_1$};
    \node at (1,-0.3) {$\mathbf{v}_2$};
    \node at (2,-0.3) {$\cdots$};
    \node at (3,-0.3) {$\mathbf{v}_{n-1}$};
    \draw [ultra thick] (0,0) -- (0,1);
    \draw [ultra thick] (1,0) -- (1,1);
    \draw [ultra thick] (2,0) -- (2,1);
    \draw [ultra thick] (3,0) -- (3,1);
    \fill [ultra thick] (-0.1,1) rectangle (0.1,1.2);
    \fill [ultra thick] (0.9,1) rectangle (1.1,1.2);
    \fill [ultra thick] (1.9,1) rectangle (2.1,1.2);
    \fill [ultra thick] (2.9,1) rectangle (3.1,1.2);
    \node at (0,1.45) {$\mathbf{w}_1$};
    \node at (1,1.45) {$\mathbf{w}_2$};
    \node at (2,1.45) {$\cdots$};
    \node at (3,1.45) {$\mathbf{w}_{n-1}$};
\end{tikzpicture}
\end{center}

The group $\prod_{i,j} GL(Q_{ij}) \times \prod_i GL(W_i) \times \mathbb{C}^\times_q$ acts as automorphisms on $X$, induced by its action on $\text{Rep}(\mathbf{v}, \mathbf{w})$. Here, $\mathbb{C}^\times_q$ scales the cotangent directions with weight $q$ and the symplectic form with weight $q^{-1}$. Let $\mathsf{T} = \mathbf{A} \times \mathbb{C}^\times_q$ be the maximal torus of this group.

As a symplectic resolution, $N_{\mathbf{v}, \mathbf{w}}$ admits a projective morphism to an affine variety (see Theorem 5.2.2 in \cite{G}):
\[
N_{\mathbf{v}, \mathbf{w}} \to N^0_{\mathbf{v}, \mathbf{w}} := \text{Spec}\big(\mathbb{C}[\mu^{-1}(0)]^G\big).
\]

A simple example arises from a quiver with one vertex framed by another, with vector spaces $V$ and $W$ of dimensions $\dim V = k$ and $\dim W = n$. Then:
\[
M = \text{Hom}(W, V), \quad N_{k,n} = T^* \text{Gr}_{k,n} = [\text{Hom}(V, W) \oplus \text{Hom}(W, V)] /\!\!/_{\theta} GL(V).
\]
The moment map is $\mu(A, B) = BA$, and stability requires $\{ (A, B) : \text{rank}(A) = k \}$.

This example generalizes to a type $A_n$ quiver with a single framing vertex:
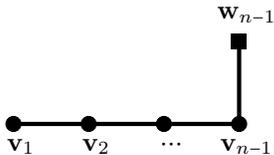
\begin{center}
\begin{tikzpicture}
    \draw [ultra thick] (0,0) -- (3,0);
    \draw [ultra thick] (3,1) -- (3,0);
    \draw [fill] (0,0) circle [radius=0.1];
    \draw [fill] (1,0) circle [radius=0.1];
    \draw [fill] (2,0) circle [radius=0.1];
    \draw [fill] (3,0) circle [radius=0.1];
    \node at (0.1,-0.3) {$\mathbf{v}_1$};
    \node at (1.1,-0.3) {$\mathbf{v}_2$};
    \node at (2.1,-0.3) {$\cdots$};
    \node at (3.1,-0.3) {$\mathbf{v}_{n-1}$};
    \fill [ultra thick] (2.9,1) rectangle (3.1,1.2);
    \node at (3.1,1.45) {$\mathbf{w}_{n-1}$};
\end{tikzpicture}
\end{center}
Here, stability demands injective maps $V_i \to V_{i+1}$ and $V_{n-1} \to W_{n-1}$, with the sequence $\mathbf{v}_1, \dots, \mathbf{v}_{n-1}, \mathbf{w}_{n-1}$ being non-decreasing for the variety to be non-empty. This corresponds to the cotangent bundle of a partial flag variety.

On $N_{\mathbf{v}, \mathbf{w}}$, we define tautological bundles:
\[
\mathcal{V}_i = L^{ss}_{\mathbf{v}, \mathbf{w}} \times_G V_i, \quad \mathcal{W}_i = L^{ss}_{\mathbf{v}, \mathbf{w}} \times_G W_i.
\]
The bundles $\mathcal{W}_i$ are topologically trivial, and tensor polynomials in $\mathcal{V}_i$, $\mathcal{W}_i$, and their duals generate the K-theory ring $K_{\mathsf{T}}(X)$, following Kirwan surjectivity (\cite{MN}). For further details, see \cite{G}, \cite{MO} (introduction), or \cite{O} (Section 4).

In localized quantum K-theory $K^{loc}_{\mathsf{T}}(N)$, with a basis of torus $\mathsf{T}$ fixed points, the eigenvalues of multiplication operators given by tautological bundles correspond to isolated solutions of the infinite $QQ$-system. Specifically, the $Q^i_+(z)$-functions, solutions to the infinite $QQ$-system, are generating functions for the eigenvalues of exterior powers of $\mathcal{V}_i$, while $\Lambda_i(z)$ generate the eigenvalues of the exterior powers of $\mathcal{W}_i$.

For the example in case of $N_{k,n} = T^* \text{Gr}_{k,n}$, the operator:
\[
\mathcal{Q}_+(z) = \sum_{i=1}^k (-1)^i z^{k-i} \Lambda^i \mathcal{V}
\]
has eigenvalues $\prod_{r=1}^k (z - a_{i_r})$ for all $k$-subsets $\{a_{i_1}, \dots, a_{i_k}\}$ labeling the fixed points of $\mathsf{T}$, forming a basis of $K^{loc}_{\mathsf{T}}(N_{k,n})$. The operator generating exterior powers of the framing bundle, $\sum (-1)^i z^i \Lambda^{n-i} \mathcal{W}$, has eigenvalue $\Lambda(z) = \prod_{i=1}^n (z - a_i)$.

What about the polynomial $Q_-(z)$? Consider the short exact sequence of bundles:
\[
0 \to \mathcal{V} \to \mathcal{W} \to q \otimes \mathcal{V}^\vee \to 0,
\]
where $q$ denotes a trivial line bundle with character $q$. The quotient bundle corresponds to the tautological bundle of $T^* \text{Gr}_{n-k,n}$, the same GIT quotient with inverted stability parameter. The $\mathcal{Q}_-(z)$-operator is then the generating function for quantum exterior powers of $\mathcal{V}^\vee$. This extends to other quiver varieties, such as type $A_n$ (see, e.g., \cite{KZ1}).

\subsection{Motivation: $QQ$-Systems, Quantum K-Theory, and Integrable Models}

The localized equivariant cohomology and K-theory of Nakajima quiver varieties form representations of the Yangian $Y_{\hbar}(\mathfrak{g}^Q)$ and quantum affine algebra $U_q(\widehat{\mathfrak{g}}^Q)$  respectively, associated with the quiver. In the simply-laced $ADE$ case, the quiver corresponds bijectively to the corresponding Dynkin diagram (\cite{nakajima1998}, \cite{Nakajima:1999hilb}, \cite{Nakajima:2001qg}, \cite{olivier1998quantum}, \cite{Vasserot:wo}, \cite{varagnolo2000quiver}).

For the cotangent bundle to the Grassmannian $N_{k,n}={\rm T^*Gr}_{k,n}$, there is an isomorphism between $\bigoplus_{k=0}^n K_{\mathsf{T}}^{loc}(N_{k,n})$ and the representation space of $U_q(\widehat{\mathfrak{sl}}_2)$:
\[
\mathcal{H} = \mathbb{C}^2(a_1) \otimes \dots \otimes \mathbb{C}^2(a_n),
\]
constructed from two-dimensional evaluation modules, with $K_{\mathsf{T}}^{loc}(N_{k,n})$ as subspaces of weight $n-2k$.

The stable basis construction \cite{MO} provides an effective realization of quantum groups in $K_{\mathsf{T}}^{loc}(N)$, enabling the construction of $R$-matrices:
\[
R_{V_i(a_i), V_j(a_j)}: V_i(a_i) \otimes V_j(a_j) \to V_j(a_j) \otimes V_i(a_i),
\]
where $\{V_i(a_i)\}$ are finite-dimensional representations of $U_q(\widehat{\mathfrak{g}}^Q)$ twisted by $a_i$, satisfying the Yang-Baxter equation. This endows the category of representations with a braided tensor category structure, with all quantum group generators expressible via $R$-matrix elements \cite{FRT}.

This allows the construction of an integrable model of spin chain type, known as XXZ model.  
Define an element in the category of finite-dimensional representations of  $U_q(\widehat{\mathfrak{g}}^Q)$, which we call \textit{physical space}, e.g., see $\mathcal{H}$ above. For an auxiliary module $W(u)$ from the same category, which depends on evaluation (spectral) parameter $u$, the \textit{transfer matrix} is (see, e.g., \cite{R}, \cite{KBI}):
\[
T_{W(u)} = \text{Tr}_{W(u)} \Big[ P R_{\mathcal{H}, W(u)} (1 \otimes Z) \Big], \quad T_{W(u)}: \mathcal{H} \to \mathcal{H},
\]
where $Z = \prod_{i=1}^r \zeta_i^{\check{\alpha}_i} \in e^{\mathfrak{h}}$ is known as {\it twist} (with $\mathfrak{h}$ the Cartan subalgebra, $\{\check{\alpha}_i\}$ simple coroots), and $P$ is the permutation operator. The object $M^Z_{W(u)} = P R_{\mathcal{H}, W(u)} (1 \otimes Z)$ is known as the \textit{quantum monodromy matrix}. The Yang-Baxter equation ensures that transfer matrices commute, forming the Bethe algebra. Their eigenvalues, describing the integrable model, satisfy the Bethe equations, which are solutions to the $QQ$-system.

$R$-matrix is an object from $U_q(\mathfrak{b}_+)\otimes U_q(\mathfrak{b}_-)$, where $\mathfrak{b}_{\pm}$ are Borel subalgebras in $\widehat{\mathfrak{g}}^Q$. This allows us to consider transfer matrices 
with $W(u)$ with representations of $U_q(\mathfrak{b}_+)$.
The $Q^i_{\pm}(z)$-functions, solutions to the $QQ$-system, arise from transfer matrices with prefundamental representations \cites{HJ, FH, FH2} of $U_q(\mathfrak{b}_+)$, where $\mathfrak{b}_\pm$ are Borel subalgebras. The operators $\mathcal{Q}^i_{\pm}(z)$ have eigenvalues $Q^i_{\pm}(z)$, satisfying $QQ$-system relations in the resulting extended Grothendieck ring.

Geometrically, $\mathcal{Q}^i_{\pm}(z)$ relate to enumerative geometry via quantum tautological classes, defined by counting \textit{quasimaps} from $\mathbb{P}^1$ to $N_{\mathbf{v}, \mathbf{w}}$. 
 The notion of quasimap (see, e.g., \cite{CKM}, \cite{O}) is deeply related to the structure of Nakajima variety as a GIT quotient of affine space and instead of maps from $\mathbb{P}^1$ to $N_{\mathbf{v}, \mathbf{w}}$, quasimaps corresponds to certain bundles and their sections with certain conditions over $\mathbb{P}^1$. 
The parameters of the deformation are known as K\"ahler parameters and the counting of quasimaps is encapsulated in the so-called vertex functions \cite{O}. They are governed \cite{O}, \cite{OS} by the difference equations well-known in the theory of representations of  quantum groups, the so-called Frenkel-Reshetikhin or quantum Knizhnik-Zamolodchikov equation \cite{FR}, where the difference parameter is given by the character of the natural multiplicative 
$\mathbb{C}^{\times}_{\bf q}$ action on $\mathbb{P}^1$. Its asymptotics (upon the limit ${\bf q}\rightarrow$ 1)  gives the eigenvalue problem for the transfer matrices, so that the K\"ahler parameters identify with the $\hat{\xi}^{-1}_i$-parameters. 

This way one can interpret $\mathcal{Q}_+^i(z)$ as operators generating exterior powers of quantum tautological bundles \cite{PSZ}, \cite{KPSZ}, while the coefficients of the twist $Z$ serve as K\"ahler parameters. The operator $\mathcal{Q}_-^i(z)$ 
also generates operators of quantum multiplication by exterior powers of quantum tautological bundles, but for the Nakajima variety with the action of simple Weyl reflection on stability parameter, and thus for K\"ahler parameter as well. In the particular example of $N_{k,n}={\rm T^*Gr}_{k,n}$, the  $\mathcal{Q}_-(z)$-operator will generate operators of quantum multiplication by exterior powers of the quantum tautological bundle on ${\rm T^*Gr}_{n-k,n}$ with the inverse K\"ahler parameter.

Altogether this provides a natural motivation for the study of the deformation of isolated solutions of the infinite $QQ$-systems as 
transitions from classical to quantum tautological class eigenvalues.

Upon certain nondegeneracy conditions one can show that such eigenvalues of $\mathcal{Q}_+^i(z)$-operators turn out to be not Puiseux series, but just power series in K\"ahler parameters. To proceed with that, it is enough to consider the monodromy around the unit disk and take into account that we have to produce the same system of eigenvalues and eigenvectors, which leads to the fact that all multivalued terms vanish. This argument, of course, relies heavily on the condition of having distinct eigenvalues.
\section{Main results}\label{mainresults}
\subsection{Deformation of solutions of infinite $qq$-systems}\label{statementfordeformationofqq-systems}
We first state our result for a simple Lie algebra $\mathfrak{g}$, and then explicitly state the result in the $\mathfrak{sl}_2$ case where we have more explicit results than in the general case.

In the notation of Remark \ref{qqfinitetoinfinite} set $t_{i}:=\xi_{i}^{-1}$.
\begin{theorem}\label{qqforgeneralg}
    Let $\{q_{\pm}^{i}(z)\}_{i=1}^{r}$ be an isolated solution of an infinite $qq$-system \eqref{infiniteqqsystem}. Then there exist $\epsilon>0$ and polynomials $\{q_{\pm}^{i,Z^{H}}(z)\}_{i=1}^{r}$ that lift $\{q_{\pm}^{i}(z)\}_{i=1}^{r}$ and solve the system \eqref{qqsystem} for all $Z^{H}\in\mathfrak{h}$ satisfying $|t_{i}|<\epsilon$ for all $i$.
\end{theorem}
We will give the proof of the above Theorem in the case of $\mathfrak{sl}_2$ and sketch the proof in the general case, which is similar to the proof in the $\mathfrak{sl}_2$-case.

Since Theorem \ref{qqforgeneralg} does not require $\Lambda_{i}(z)$'s to be separable, we obtain the following using the same arguments as in \cite{BSZ}:
\begin{corollary}
    Let $\mathfrak{g}$ be simply-laced and $Z^{H}\in\mathfrak{h}$ be regular semisimple. Let $\Lambda_{1}(z),\ldots,\Lambda_{r}(z)$ be arbitrary monic polynomials. Assume that $\big(\{\deg(\Lambda_{i})\}_{i=1}^{r},\{\deg(q^{i,Z^{H}}_{+})\}_{i=1}^{r},\varnothing\big)$ is an admissible combinatorial data. Then there is a one-to-one correspondence between the nondegenerate $Z^H$-twisted Miura opers and the solutions of the Bethe Ansatz equations \eqref{betheqqg}.
\end{corollary}
\begin{proof}
    It is enough to note that the admissibility of the data $\big(\{\deg(\Lambda_{i})\}_{i=1}^{r},\{\deg(q^{i,Z^{H}}_{+})\}_{i=1}^{r},\varnothing\big)$ implies that the limit of $\{q_{\pm}^{i,Z^{H}}(z)\}_{i=1}^{r}$ as $\langle\alpha_{i},Z^{H}\rangle\to \infty$ for $1\leq i\leq r$, is an isolated solution of the corresponding infinite $qq$-system \eqref{infiniteqqsystem}. The rest of the proof is the same as in \cite{BSZ}.
\end{proof}
Using tropical geometry in the case where the roots of $q_{\pm}^{i}(z)$'s satisfy a generic condition, we get the following stronger version of Theorem \ref{qqforgeneralg}.
\begin{theorem}\label{analyticityqq}
    Let $\{q_{\pm}^{i}(z)\}_{i=1}^{r}$ be an isolated solution of an infinite $qq$-system \eqref{infiniteqqsystem} such that $q_{\pm}^{i}(0)\neq 0$ for all $i$. Then there exist $\epsilon>0$ and polynomials $\{q_{\pm}^{i,Z^{H}}(z)\}_{i=1}^{r}$ that lift $\{q_{\pm}^{i}(z)\}_{i=1}^{r}$ and solve the system \eqref{qqsystem} for all $Z^{H}\in\mathfrak{h}$ satisfying $|t_{i}|<\epsilon$ for all $i$. Moreover, the polynomials $\{q_{\pm}^{i,Z^{H}}(z)\}_{i=1}^{r}$ are Puiseux series in $t_{1},\ldots, t_{r}$ and therefore depend analytically on $t_{1},\ldots,t_{r}$.
\end{theorem}
To state the next result and for later use,
let us explicitly formulate Theorem \ref{qqforgeneralg} for $\mathfrak{sl}_2$.
Let $q_{+}^{t}(z)=\prod_{i=1}^{m}(z+x_{i}(t))$, $q_{-}^{t}(z)=\prod_{j=1}^{n}(z+y_{j}(t))$ and $\Lambda(z)=\prod_{k=1}^{l}(z+a_{k})^{m_{k}}$, where $t$ is a formal parameter, the $a_{k}\in\mathbb{C}$, $k=1,\ldots,l$, $a_{k}$'s are distinct, and $\sum_{k=1}^{l}m_{k}=m+n$. Note that $\Lambda(z)$ does not depend on the parameter $t$. Substituting $q_{+}^{t}(z)$, $q_{-}^{t}(z)$  and $\Lambda(z)$ in the equation of the finite $qq$-system \eqref{sl2qq} with parameter $t$, we get the following set of equations upon comparing the coefficients of $z^{m+n-k}$, $k=1,\ldots,m+n$:
\[
f_{k}:=e_{k}(x_{1},\ldots,x_{m},y_{1},\ldots,y_{n})+tp_{k-1}(x_{1},\ldots,x_{m},y_{1},\ldots,y_{n})-d_{k}=0,
\]
where $e_{k}(x_{1},\ldots,x_{m},y_{1},\ldots,y_{n})$ is the elementary symmetric polynomial of degree $k$, $p_{k-1}$ is a polynomial in $x_{1},\ldots,x_{m},y_{1},\ldots,y_{n}$ with integer coefficients of degree $k-1$ and $d_{k}$ is the coefficient of $z^{m+n-k}$ in $\Lambda(z)$.

When $\mathfrak{g}=\mathfrak{sl}_{2}$, we have the following stronger statement.
\begin{theorem}\label{sl2analyticity}
    Assume that $\Lambda(0)\neq 0$. For every given choice $(x_{1}(0),$ $\ldots,x_{m}(0),y_{1}(0),\ldots,y_{n}(0))$ of a solution of an infinite $qq$-system \eqref{sl2infiniteqq}, there exists $\epsilon>0$ and analytic functions $x_{1}(t),\ldots,x_{m}(t),y_{1}(t),$ $\ldots,y_{n}(t)$ that solve the finite $qq$-system \eqref{sl2qq} with parameter $t$ for all $t\in\mathbb{C}$ such that $|t|<\epsilon$. Moreover, in this way we get all the solutions of the system \eqref{sl2qq} (and the corresponding Bethe equations \eqref{bethesl2}) for small enough $t$.
\end{theorem}
The evidence in support of the existence of a solution for theorem \ref{sl2analyticity} comes from our computations on Gfan \cite{gfan} (a software for doing calculations in tropical geometry) in the cases when $m$ and $n$ are small. In the general case, the idea is to use the implicit function theorem to find the deformations and use the fundamental theorem of tropical algebraic geometry (see \cite{MS}) to show their analyticity. To the best of our knowledge, tropical geometry has not been used in the study of $qq$-systems before. 

Recall the field $L_{c}$ of convergent Puiseux series from Section \ref{basicsonpuiseuxseries} and assume $\Lambda(0)\neq 0$. Consider the variety $\Def_{qq}:=V\big(\{f_{1},\ldots,f_{m+n}\}\big)\subset(L_{c}^{\times})^{m+n}$. Based on the computations on Gfan in the small degree cases, we have the following theorem.
\begin{theorem}\label{tropicalqq}
     Assume that all the coefficients $d_{i}$'s of $\Lambda(z)$ are non-zero. Then set of polynomials $f_{1},\ldots,f_{m+n}$  forms a tropical basis of the tropical variety $\trop(\Def_{qq})$ (see definition \ref{definitionoftropicalvariety}) and $\trop(\Def_{qq})=\{(0,\ldots,0)\}$. 
\end{theorem}
We give the proof of theorem \ref{tropicalqq} in Section \ref{proofqq}.
\subsection{Deformation of solutions of infinite $QQ$-systems}\label{statementforQQsystems}
We first state our result for a connected, simply connected, simple algebraic group $G$ and then we will explicitly state the result in the $SL_2$ where we have more explicit results than the general case.

In the notation of Remark \ref{QQfinitetoinfinite} set $t_{i}:=\hat{\xi}_{i}^{-1}$.
\begin{theorem}\label{QQforgeneralg}
    Let $\{Q_{\pm}^{i}(z)\}_{i=1}^{r}$ be an isolated solution of an infinite $QQ$-system \eqref{infiniteQQsystem}. Then there exists $\epsilon>0$ and polynomials $\{Q_{\pm}^{i,Z}(z)\}_{i=1}^{r}$ that lifts $\{Q_{\pm}^{i}(z)\}_{i=1}^{r}$ and solve the system \eqref{QQsystem} for all $Z\in H$ satisfying $|t_{i}|<\epsilon$ for all $i$.
\end{theorem}
\begin{remark}
    The assumptions of Theorem \ref{QQforgeneralg} are satisfied for example, when we consider the cotangent bundles to partial flag varieties viewed as a Nakajima quiver variety of type $A$, the $\Lambda_{i}(z)$'s are generating functions of multiplication operators by exterior powers of tautological bundles corresponding to the framing vertices and $Q^{i}_{+}(z)$'s are generating functions of multiplication operators by exterior powers of tautological bundles corresponding to the original vertices (see Section \ref{Motivation} and \cite{KSZ} for the relation of $QQ$-systems in type $A$ with equivariant $K$-theory of Nakajima quiver varieties of type $A$). 
\end{remark}
We will give a proof of the above theorem in the case of $SL_2$. The proof in the general case is quite similar.

Using tropical geometry in the case where the roots of $Q_{\pm}^{i}(z)$'s satisfy a generic condition, we get the following stronger version of Theorem \ref{QQforgeneralg}.
\begin{theorem}\label{analyticityQQ}
    Let $\{Q_{\pm}^{i}(z)\}_{i=1}^{r}$ be an isolated solution of an infinite $QQ$-system \eqref{infiniteQQsystem} such that $Q_{\pm}^{i}(0)\neq 0$ for all $i$. Then there exists $\epsilon>0$ and polynomials $\{Q_{\pm}^{i,Z}(z)\}_{i=1}^{r}$ that lifts $\{Q_{\pm}^{i}(z)\}_{i=1}^{r}$ and solve the system \eqref{QQsystem} for all $Z\in H$ satisfying $|t_{i}|<\epsilon$ for all $i$. Moreover, the polynomials $\{Q_{\pm}^{i,Z}(z)\}_{i=1}^{r}$ are Puiseux series in $t_{1},\ldots, t_{r}$ and therefore depend analytically on $t_{1},\ldots,t_{r}$.
\end{theorem}
To state the next result and for later use,
let us explicitly formulate Theorem \ref{QQforgeneralg} for $SL_2$.

Let $Q_{+}^{t}(z)=\prod_{i=1}^{m}(z+x_{i}(t))$, $Q_{-}^{t}(z)=\alpha(t)\prod_{j=1}^{n}(z+y_{j}(t))$ and $\Lambda(z)=\prod_{k=1}^{l}(z+a_{k})^{m_{k}}$, where $t$ is a formal parameter, $a_{k}\in\mathbb{C}$, $k=1,\ldots,l$, $a_{k}$'s are distinct and $\sum_{k=1}^{l}m_{k}=m+n$. As in the differential case $\Lambda(z)$ does not depend on the parameter $t$. Substituting $Q_{+}^{t}(z)$, $Q_{-}^{t}(z)$  and $\Lambda(z)$ in the equation of the finite $QQ$-system \eqref{sl2QQ} with parameter $t$, we get (recall that $\Lambda(z)$ is monic)
\[
\alpha(t)=\frac{1}{q^{m}-tq^{n}}
\]
and the following set of equations upon comparing the coefficients of $z^{m+n-k}$, $k=1,\ldots,m+n$:
\[
g_{k}:=\frac{q^m}{q^{m}-tq^{n}} e_{k}\bigg(\frac{x_{1}}{q},\ldots,\frac{x_{m}}{q},y_{1},\ldots,y_{n}\bigg)
-
\frac{tq^n}{q^{m}-tq^{n}}
e_{k}\bigg(x_{1},\ldots,x_{m},\frac{y_{1}}{q},\ldots,\frac{y_{n}}{q}\bigg)
-
d_{k}
=
0,
\]
where $e_{k}$ is the elementary symmetric polynomial of degree $k$ and $d_{k}$ is the coefficient of $z^{m+n-k}$ in $\Lambda(z)$.

When $G=SL_{2}$, we have the following stronger statement.
\begin{theorem}\label{SL2analyticity}
    Assume that $\Lambda(0)\neq 0$. For every choice $(x_{1}(0),\ldots,x_{m}(0),y_{1}(0),\ldots,y_{n}(0))$ of a solution of the infinite $QQ$-system \eqref{infinitesl2QQ}, there exists $\epsilon>0$ and analytic functions $x_{1}(t),\ldots,x_{m}(t),y_{1}(t),$ $\ldots,y_{n}(t)$ that solve the finite $QQ$-system \eqref{sl2QQ} with parameter $t$ for all $t\in\mathbb{C}$ such that $|t|<\epsilon$. Moreover, in this way  we get all solutions of the system \eqref{sl2QQ} (and the corresponding Bethe equations \eqref{betheSL2}) for small enough $t$.
\end{theorem}
\begin{remark} There are various versions \cite{KZ} of $QQ$-systems. For example, there is an additive version of the finite (resp. infinite) $QQ$-system, which is defined using the additive shift $z\mapsto z+\hbar$ in place of $z\mapsto qz$ in the system \eqref{QQsystem} (resp. \eqref{infiniteQQsystem}). The methods used in the case of $QQ$-systems for deformations of solutions works quite similarly for these versions.
\end{remark}
Assume now that $\Lambda(0)\neq 0$. Consider the variety $\Def_{QQ}:=V\big(\{g_{1},\ldots,g_{m+n}\}\big)\subset(L_{c}^{\times})^{m+n}$.
Motivated by the differential case, we have the following result.
\begin{theorem}\label{tropicalQQ}
     Assume that all the coefficients $d_{i}$'s of $\Lambda(z)$ are non-zero. Then set of polynomials $g_{1},\ldots,g_{m+n}$  forms a tropical basis of the tropical variety $\trop(\Def_{QQ})$ (see definition \ref{definitionoftropicalvariety}) and $\trop(\Def_{QQ})=\{(0,\ldots,0)\}$. 
\end{theorem}
We give the proof of theorem \ref{tropicalQQ} in Section \ref{proofQQ}.
\section{Tropical geometry}\label{tropicalgeometry}
In this section we recall some facts from tropical geometry. We will use them in the next two sections to give proofs of Theorems \ref{qqforgeneralg}, \ref{analyticityqq}, \ref{sl2analyticity}, \ref{tropicalqq}, \ref{QQforgeneralg}, \ref{analyticityQQ}, \ref{SL2analyticity} and \ref{tropicalQQ}.

We will denote by $K$ a valued field with valuation $val$ and its residue field by $\mathds{k}$. Let $\Gamma_{val}$ denote the value group of $K$. We assume that $val:K\backslash\{0\}\rightarrow \Gamma_{val}$ has a splitting, that is, there exists a map $\Gamma_{val}\rightarrow K\backslash{0}$, $w\mapsto t^{w}$ such that $val(t^{w})=w$. For any $a\in K$ that lies in the valuation ring of $K$, we denote by $\overline{a}$ the image of $a$ in the residue field $\mathds{k}$.
\begin{definition}
    Let $f\in K[x_{1}^{\pm},\ldots,x_{n}^{\pm}]$ be a Laurent polynomial, write $f=\sum_{\bm{u}\in\mathbb{Z}^{n}}c_{\bm{u}}x^{\bm{u}}$.
    \begin{enumerate}
        \item The tropical hypersurface $\trop(V(f))$ is the set
    \[
    \{\bm{w}\in\mathbb{R}^{n}:\text{ the minimum in }\trop(f)(\bm{w})\text{ is achieved at least twice}\},
    \]
where $\trop(f)(\bm{w})=min_{\bm{u}\in\mathbb{Z}^{n}}(val(c_{\bm{u}})+\sum_{i=1}^{n}u_{i}w_{i})$.
\item For $\bm{w}\in\mathbb{R}^n$, the initial form $in_{\bm{w}}(f)\in\mathds{k}[x_{1}^{\pm},\ldots,x_{n}^{\pm}]$ is defined as 
\[
in_{\bm{w}}(f)=\sum_{\bm{u}:val(c_{\bm{u}})+\bm{w}\cdot\bm{u}=\trop(f)(\bm{w})}\overline{t^{-val(c_{\bm{u}})}c_{\bm{u}}}x^{\bm{u}}.
\]
\item Let $I$ be an ideal in $K[x_{1}^{\pm},\ldots,x_{n}^{\pm}]$. For $\bm{w}\in\mathbb{R}^n$, the initial ideal $in_{\bm{w}}(I)$ is the ideal in $\mathds{k}[x_{1}^{\pm},\ldots,x_{n}^{\pm}]$ generated by the initial forms $in_{\bm{w}}(f)$ for all $f\in I$.
\end{enumerate}
\end{definition}
\begin{example}
Recall the field $L$ of Puiseux series from Section \ref{basicsonpuiseuxseries}.  Let $f=(2t+t^{3})x_{1}x_{2}-4t^{3}x_{3}+t^{8}x_{2}x_{3}\in L[x_{1}^{\pm},x_{2}^{\pm},x_{3}^{\pm}]$. If $\bm{w}=(0,0,0)$, then $\trop(f)(\bm{w})=1$ and $in_{\bm{w}}(f)=2x_{1}x_{2}$. If $\bm{w}=(1,1,-2)$, then $\trop(f)(\bm{w})=1$ and $in_{\bm{w}}(f)=-4x_{3}$.    
\end{example}
We will need the following fact (see \cite{MS}*{Lemma 2.6.2 (1)}):
\begin{fact}\label{fact1}
    Let $I$ be an ideal in $K[x_{1}^{\pm},\ldots,x_{n}^{\pm}]$ and $\bm{w}\in\mathbb{R}^{n}$. If $g\in in_{\bm{w}}(I)$, then $g=in_{\bm{w}}(h)$ for some $h\in I$.
\end{fact}
\begin{definition}\label{definitionoftropicalvariety}
    Let $I$ be an ideal in $K[x_{1}^{\pm},\ldots,x_{n}^{\pm}]$ and let $X=V(I)$ be its variety in the algebraic torus $T^{n}\cong(K^{\times})^{n}$. The tropical variety associated to $X$ is defined to be the following intersection of tropical hypersurfaces:
    \[
    \trop(X)=\bigcap_{f\in I}\trop(V(f)).
    \]
\end{definition}
We recall the following result (see \cite{MS}*{Theorem 3.2.3}), known as the Fundamental theorem of tropical algebraic geometry.
\begin{fact}\label{ftotag}
    Let $K$ be an algebraically closed field with nontrivial $val$, let $I$ be an ideal in $K[x_{1}^{\pm},\ldots,x_{n}^{\pm}]$ and let $X=V(I)$ be its variety in the algebraic torus $T^{n}\cong(K^{\times})^{n}$. Then the following three subsets of $\mathbb{R}^{n}$ coincide:
    \begin{enumerate}[label=(\roman*)]
        \item the tropical variety $\trop(X)$;
        \item the set of all vectors $\bm{w}\in\mathbb{R}^{n}$ with $in_{\bm{w}}(I)\neq\langle1\rangle$;
        \item the closure of the set of coordinatewise valuations of points in $X$,
        \[
        val(X)=\{(val(y_{1}),\ldots,val(y_{n})):(y_{1},\ldots,y_{n})\in X\}
        \]
    \end{enumerate}
\end{fact}
\begin{remark}\label{kapranov}
\begin{enumerate}[label=(\roman*)]
    \item In the proof of Fact \ref{ftotag}, it is shown that the set of vectors in $(ii)$ lying in $\Gamma^{n}_{val}$ is equal to $val(X)$. 
    \item In the special case when $X$ is defined by a single Laurent polynomial $f$, Fact \ref{ftotag} was proved by Kapranov (see \cite{EKL}) and in this case, we have
    \[
    \trop(X)=\{\bm{w}\in\mathbb{R}^n:in_{\bm{w}}(f)\neq\langle 1\rangle\}.
    \]
 \end{enumerate}
\end{remark}
\begin{definition}
    A finite generating set $\mathcal{T}$ for an ideal $I\subset K[x_{1}^{\pm},\ldots,x_{n}^{\pm}]$ is said to be a tropical basis of $I$ if 
    \[ \trop(V(I))=\bigcap_{f\in\mathcal{T}}\trop(V(f)).
    \]
\end{definition}
Recall that a minimal associated prime of an ideal $I$ in a commutative ring $R$ is a prime ideal of $R$ containing $I$ and is minimal with this property. Hence minimal associated primes correspond to the irreducible components of $Spec(R/I)$.  The following result (\cite{MS}*{Lemma 3.2.6}) will give us that the variety of deformations is finite.
\begin{fact}\label{finiteness}
    Let $X\subset T^n$ be an irreducible variety of dimension $d$, with prime ideal $I\subset K[x_{1}^{\pm},\ldots,x_{n}^{\pm}]$, and let $\bm{w}\in \trop(X)\cap\Gamma^{n}_{val}$. Then all minimal associated primes of the initial ideal $in_{\bm{w}}(I)$ in $\mathds{k}[x_{1}^{\pm},\ldots,x_{n}^{\pm}]$ have the same dimension $d$.
\end{fact}
\begin{corollary}\label{finitenessofdef}
    Suppose $\Lambda(z)$ is a monic polynomial such that all its coefficients are non-zero. Then varieties $\Def_{qq}$ and $\Def_{QQ}$ are finite. In particular, the associated tropical varieties are finite.
\end{corollary}
\begin{proof}
    We give an argument for $\Def_{qq}$. The argument for $\Def_{QQ}$ is similar. It is enough to show that each irreducible component of $\Def_{qq}$ has dimension $0$. The claim now follows from Theorem \ref{tropicalqq} and the fact that there are only finitely many solutions of the infinite $qq$-system. 
\end{proof}
\begin{remark}
    Corollary \ref{finitenessofdef} also follows in the differential (resp. difference) case from Theorem \ref{tropicalqq} (resp. Theorem \ref{tropicalQQ}) and the fact that a variety $X\subset T^{n}$ is a finite set of points if the tropical variety $\trop(X)$
is a finite set (see \cite{MS}*{Lemma 3.3.9}).
\end{remark}
We will also need the notion of the Hahn series in the proofs of Theorems \ref{analyticityqq} and \ref{analyticityQQ}, which we recall now.
\begin{definition} Let $\Gamma$ be an ordered group. 
    The field of Hahn series $\mathbb{C}[[t^{\Gamma}]]$ with coefficients in $\mathbb{C}$ and with value group $\Gamma$ consists of formal expressions of the form
    \[
    f=\sum_{e\in\Gamma}c_{e}t^{e}
    \]
    with $c_{e}\in\mathbb{C}$ such that the support $Supp(f):=\{e\in\Gamma:c_{e}\neq 0\}$ of $f$ is a well-ordered subset of $\Gamma$.
The usual operations of addition and multiplication makes $\mathbb{C}[[t^{\Gamma}]]$ into a field.
\end{definition}
 The natural valuation $$\nu:\mathbb{C}[[t^{\Gamma}]]\backslash \{0\}\rightarrow \Gamma,
 $$
 \[
 f\mapsto \min_{e\in Supp(f)}e
 \]
 makes $\mathbb{C}[[t^{\Gamma}]]$ into a valued field. We define the valuation of $0\in \mathbb{C}[[t^{\Gamma}]]$ as $+\infty$.

\begin{fact}\label{Hahnseries}
    \begin{enumerate}[label=(\roman*)]
        \item If $\Gamma$ is divisible, then $\mathbb{C}[[t^{\Gamma}]]$ is an algebraically closed field.
        \item In particular, if $\Gamma=\alpha_{1}\mathbb{Q}+\ldots+\alpha_{r}\mathbb{Q}$ for some $\alpha_{1},\ldots,\alpha_{r}\in\mathbb{R}$, then $\mathbb{C}[[t^{\Gamma}]]$ is algebraically closed.
    \end{enumerate}
\end{fact}
\section{\texorpdfstring{$qq$}{qq}-systems}
\subsection{Proofs}\label{proofqq}
In this section, we give proofs of Theorems \ref{qqforgeneralg}, \ref{analyticityqq}, \ref{sl2analyticity} and \ref{tropicalqq}.

We first give the proof of Theorem \ref{sl2analyticity}. Then we will sketch the proofs of Theorems \ref{qqforgeneralg} and \ref{analyticityqq}.
\begin{proof}
    (of Theorem \ref{sl2analyticity})
For ease of notation, let us write $\big(x_{1}(0),\ldots,x_{m}(0),y_{1}(0),\ldots,y_{n}(0)\big)=(b_{1},\ldots,b_{m+n})$. Let us suppose that $l$ of the $b_i$'s are distinct, say they are $b_{j_{1}},\ldots,b_{j_{l}}$ and the remaining ones $b_{i_{1}},\ldots,b_{i_{m+n-l}}$ are repeated.

Recall the polynomials $f_{k}$, $k=1,\ldots,m+n$ from Section \ref{statementfordeformationofqq-systems}. Let us define the following map:
\[
\phi:\mathbb{C}^{m}\times\mathbb{C}^{n}\times\mathbb{C}\rightarrow\mathbb{C}^{m+n}
\]
\[
(x_{1},\ldots,x_{m},y_{1},\ldots,y_{n},t)\mapsto(f_{1},\ldots,f_{m+n})
\]
Then $\phi(b_{1},\ldots,b_{m+n},0)=(0\ldots,0)$ as $(b_{1},\ldots,b_{m+n})$ is a solution of the infinite $qq$-system. The derivative matrix of $\phi$ at $\underline{a}=(b_{1},\ldots,b_{m+n},0)$ about $x_{1},\ldots,x_{m}$, $y_{1},\ldots,y_{n}$ is given by
\[
A=
\begin{bmatrix}
1
&
1
&
\cdots
&
1
\\
e_{1}\big(b_{2},b_{3},\ldots,b_{m+n}\big)
&
e_{1}\big(b_{1},b_{3},\ldots,b_{m+n}\big)
&
\cdots
&
e_{1}\big(b_{1},b_{2},\ldots,b_{m+n-1}\big)
\\
\vdots & \vdots & \ddots & \vdots
\\
e_{m+n-1}\big(b_{2},b_{3},\ldots,b_{m+n}\big)
&
e_{m+n-1}\big(b_{1},b_{3},\ldots,b_{m+n}\big)
&
\cdots
&
e_{m+n-1}\big(b_{1},b_{2},\ldots,b_{m+n-1}\big)
\end{bmatrix}
\]
We need a lemma.
\begin{lemma}\label{rank}
    $\rank(A)=l$.
\end{lemma}
\begin{proof}
    The entries $a_{i,j}$ of $A$ are the coefficients of $z^{m+n-i}$ in the polynomial $\Lambda(z)/(z+b_{j})$. It follows that the columns $C_{j_{1}},\ldots,C_{j_{l}}$ are linearly independent and this is a maximal linearly independent set as any other column of $A$ is one of these. 
\end{proof}
Let us complete the set of columns $\{C_{j_{1}},\ldots,C_{j_{l}}\}$ in Lemma \ref{rank} to a basis of $\mathbb{C}^{m+n}$, say $\{C_{j_{1}},\ldots,C_{j_{l}},$ $\tilde{C}_{1},\ldots,\tilde{C}_{m+n-l}\}$ and let the $k$-th coordinate of $\tilde{C}_{i}$ be denoted by $\tilde{c}_{k,i}$, $1\leq i\leq m+n-l$ and $1\leq k\leq m+n$.

For each $1\leq i\leq m+n-l$, let us introduce a new variable $s_{i}$ and consider the new polynomials $h_{k}$, $1\leq k\leq m+n$, which are defined as:
\[ 
h_{k}:=f_{k}+\tilde{c}_{k,1}s_{1}+\ldots+\tilde{c}_{k,m+n-l}s_{m+n-l}
\]
and define the map $\tilde{\phi}$ as:
\[
\tilde{\phi}:\mathbb{C}^{m}\times\mathbb{C}^{n}\times\mathbb{C}\times\mathbb{C}^{m+n-l}\rightarrow\mathbb{C}^{m+n}
\]
\[
(x_{1},\ldots,x_{m},y_{1},\ldots,y_{n},t,s_{1},\ldots,s_{m+n-l})\mapsto(h_{1},\ldots,h_{m+n})
\]
Let $\underline{\tilde{a}}=(b_{1},\ldots,b_{m+n},0,0,\ldots,0)$. Then $\tilde{\phi}(\underline{\tilde{a}})=(0,\ldots,0)$ and the derivative matrix of $\tilde{\phi}$ at $\underline{\tilde{a}}$ about the $j_{1}$-st, $\ldots$, $j_{l}$-th variables and $s_{1},\ldots, s_{m+n-l}$ has rank $m+n$.

By the implicit function theorem, there exists a polydisc $D_{1}$ around $(b_{i_{1}},\ldots,b_{i_{m+n-l}},0)$, a polydisc $D_{2}$ around $(b_{j_{1}},\ldots,b_{j_{l}},0,\ldots,0)$ and analytic functions
\[
\gamma_{i}:D_{1}\rightarrow\mathbb{C},\quad 1\leq i\leq m+n
\]
such that
\[
\gamma:=(\gamma_{1},\ldots,\gamma_{m+n}):D_{1}\rightarrow D_{2}\subset\mathbb{C}^{m+n}
\]
and 
$\tilde{\phi}$ is $0$ at point in $D_{1}\times D_{2}$ if and only if it lies on the graph of $\gamma$.

Now let us consider the variety $X:=V(\{h_{1},\ldots,h_{m+n}\})\subset\mathbb{C}^{2m+2n+1-l}$, that is, we are now viewing $h_{i}$'s as polynomials in the ring $\mathbb{C}[x_{1},\ldots,x_{m},y_{1},\ldots,y_{n},t,s_{1},\ldots,s_{m+n-l}]$.
Then by the above argument, there exists an irreducible component $Y$ of $X$ containing $\underline{\tilde{a}}$ of dimension $\geq m+n-l+1$.

Consider the subvariety $Z:=V(\{s_{1},\ldots,s_{m+n-l}\})\subset\mathbb{C}^{2m+2n+1-l}$. Then every irreducible component of $Y\cap Z$ has dimension $\geq 1$ by the following result (\cite{Har}*{Proposition 1.7, Chapter I}), known as the affine dimension theorem:
\begin{fact}\label{dimensiontheorem}
    Let $Y$, $Z$ be irreducible varieties of dimensions $r$, $s$ in $\mathbb{A}^{n}$. Then every irreducible component $W$ of $Y\cap Z$ has dimension $\geq r+s-n$.
\end{fact}
Since $\underline{\tilde{a}}\in Y\cap Z$, let us take an irreducible component $C$ containing $\underline{\tilde{a}}$, which has dimension $\geq 1$ by Fact \ref{dimensiontheorem}. Consider the projection of $C$ onto the $t$-coordinate:
\[
pr:C\subset\mathbb{C}^{2m+2n+1-l}
\rightarrow
\mathbb{C}.
\]
We have the following lemma.
\begin{lemma}\label{path}
 The image of $pr$, $pr(C)$ contains an open set (in analytic topology) around $0$. 
\end{lemma}
\begin{proof}
This follows by observing that $pr(C)$ is a connected constructible subset of $\mathbb{C}$, $pr^{-1}(\{0\})$ is finite and $\dim(C)\geq 1$.
\end{proof}
Note that so far it is not clear that $x_{i}$'s and $y_{j}$'s are analytic functions of $t$ alone.
This is the content of the next proposition.
\begin{proposition}\label{limitingargument}
 There exists $\epsilon>0$ and analytic functions $x_{1}(t),\ldots,x_{m}(t),y_{1}(t),$ $\ldots,y_{n}(t)$ that solve the finite $qq$-system \eqref{sl2qq} with parameter $t$ for all $t\in\mathbb{C}$ such that $|t|<\epsilon$.  
\end{proposition}
\begin{proof}
    Recall the variety $\Def_{qq}=V\big(\{f_{1},\ldots,f_{m+n}\}\big)\subset (L_{c}^{\times})^{m+n}$ from Section \ref{statementfordeformationofqq-systems}. It suffices to show that for $\bm{0}=(0,\ldots,0)$, we have $in_{\bm{0}}(I)\neq\langle1\rangle$, where $I$ is the ideal generated by $f_{1},\ldots,f_{m+n}$ in the ring of Laurent polynomials $\mathcal{R}:=L_{c}[x_{1}^{\pm},\ldots,x_{m}^{\pm},y_{1}^{\pm},\ldots,y_{n}^{\pm}]$. Suppose that is not the case, then by Fact \ref{fact1} there exists $g_{i}\in\mathcal{R}$, $1\leq i\leq m+n$, $p_{1},p_{2}\in\mathcal{R}$ such that $p_{1}$ is a non-zero monomial in $x_{1}^{\pm},\ldots,x_{m}^{\pm},y_{1}^{\pm},\ldots,y_{n}^{\pm}$ and
\begin{equation}\label{equation}
g_{1}f_{1}+\ldots+g_{m+n}f_{m+n}=p_{1}+t^{v}p_{2},  
\end{equation}
where $v\in\mathbb{Q}_{>0}$.
Note that on $C$, the $f_{i}$'s can be replaced by $h_{i}$'s in the above equation. By Lemma \ref{path}, there exists a sequence $(t_{k})_{k\in\mathbb{N}}$ converging to $0$ and a sequence $(c_{k})_{k\in\mathbb{N}}$ in $C$ that is a lift of $(t_{k})_{k\in\mathbb{N}}$ under $pr$ such that the LHS of \eqref{equation} is $0$ when evaluated at $c_{k}$'s. By the continuity of $\gamma$, $p_{1}(b_{1},\ldots,b_{m+n})=0$, which is a contradiction to our assumption about the roots of $\Lambda(z)$. This completes the proof of the proposition and Theorem \ref{sl2analyticity}.
\end{proof}
\end{proof}
\begin{proof}
(of Theorem \ref{qqforgeneralg})
    The argument in this case is similar to the proof in the $\mathfrak{sl}_2$ case. Here we work with $r$ parameters $t_{1},\ldots,t_{r}$ instead of a single parameter $t$ and $pr:C\rightarrow\mathbb{C}^{r}$. The only changes are in the proof of Lemma \ref{path}, where we now use upper semi-continuity of the dimensions of the fibres of morphisms of varieties, and $pr^{-1}(0,\ldots,0)$ is finite (atleast locally around the solution $\{q^{i}_{\pm}(z)\}_{i=1}^{r}$) due to the assumptions on $q^{i}_{\pm}(z)$'s, that is, $q^{i}_{\pm}(z)$ is an isolated solution of an infinite $qq$-system \eqref{infiniteqqsystem}.
\end{proof}
We follow \cite{Yu} for the next proof.
\begin{proof}
    (of Theorem \ref{analyticityqq})
    Choose $\alpha_{1},\ldots,\alpha_{r}\in\mathbb{R}_{>0}$ that are linearly independent over $\mathbb{Q}$. Let $\Gamma=\alpha_{1}\mathbb{Q}+\ldots+\alpha_{r}\mathbb{Q}$. We can map the ring of multivariate Puiseux series in $t_{1},\ldots,t_{r}$ injectively into $\mathbb{C}[[t^{\Gamma}]]$ via:
    \[
    t_{i}\mapsto t^{\alpha_{i}},\quad 1\leq i\leq r.
    \]
    By Fact \ref{Hahnseries}, $\mathbb{C}[[t^{\Gamma}]]$ is an algebraically closed valued field. Applying the fundamental theorem of tropical geometry to $\mathbb{C}[[t^{\Gamma}]]$ in the argument of Proposition \ref{limitingargument}, we see that there is a solution with valuation $0$ in  $\mathbb{C}[[t^{\Gamma}]]$. The theorem now follows by lifting this Hahn series-valued solution back to the ring of multivariate Puiseux series (the fact that the exponents of this lift have bounded denominators follow from \cite{AI}*{Theorem 1}: the coefficients of the equations defining the $qq$-system \eqref{finiteqq} have coefficients in the field of $\omega$-positive Puiseux series \cite{AI}*{Section 2} for $\omega=(\alpha_{1},\ldots,\alpha_{r})$ and the field of $\omega$-positive series is algebraically closed \cite{AI}*{Theorem 1}).
\end{proof}
\begin{proof}
(of Theorem \ref{tropicalqq})
  Clearly, we have
  \[
  \trop(\Def_{qq})\subset\bigcap_{i=1}^{m+n}\trop(V(f_{i})).
  \]
  By Theorem \ref{analyticityqq}, $\bm{0}\in\trop(\Def_{qq})$. 
  We will now show $\bm{0}$ is the only element of the RHS in the above set containment. Let $\bm{w}=(w_{1},\ldots,w_{m+n})\in\mathbb{R}^{m+n}$ be an arbitrary element of RHS. Since $\bm{w}\in\trop(V(f_{1}))$, using Remark \ref{kapranov}$(ii)$ and our assumption that all the coefficients $d_{i}$'s of $\Lambda(z)$ are non-zero,  we have the following possibilities:
  \begin{enumerate}[label=(\roman*)]
      \item $w_{i_{0}}=0$ for some $i_{0}$ and $w_{i}\geq 0$ for other $i$'s, \text{or}
      \item $w_{i_{0}}=w_{j_{0}}<0$ for some $i_{0}\neq j_{0}$ and $w_{i}\geq w_{i_{0}}$ for other $i$'s.
  \end{enumerate}
 In case $(i)$ let us consider $\trop(V(f_{2}))$, we must have $w_{i_{0}}+w_{j}=0$ for some $j\neq i_{0}$. Thus, $w_{j}=0$ in this case. Now iterating this procedure with other $f_{k}$'s, it is clear that after the $k$-th step, atleast $k$ coordinates of $\bm{w}$ are zero. This gives, $\bm{w}=\bm{0}$.

 In case $(ii)$ let us consider $\trop(V(f_{2}))$, then we must have $w_{i_{0}}+w_{j}=w_{i_{0}}+w_{j_{0}}$ for some $j\neq i_{0}, j_{0}$. Thus, $w_{j}=w_{i_{0}}<0$ in this case. Now iterating this procedure with other $f_{k}$'s, it is clear that after the $k$-th step, $1\leq k\leq m+n-1$, atleast $k+1$ coordinates of $\bm{w}$ are negative. In particular, after the $(m+n-1)$-th step, all entries of $\bm{w}$ are strictly negative. This gives a contradiction since $in_{\bm{w}}(f_{m+n})$ is a non-monomial by our assumption on $\bm{w}$. This completes the proof of Theorem \ref{tropicalqq}.
\end{proof}
\subsection{Inhomogeneous Gaudin model} \label{applicationsl2}
In this section, we give an application of Theorem \ref{sl2analyticity}, Theorem \ref{tropicalqq} to the Bethe ansatz equations of the $\mathfrak{sl}_2$ inhomogeneous Gaudin model (\cite{FFR_Opers},\cite{FFL_irreg}).
\begin{definition}
    Let $y_{1}(z),\ldots, y_{r}(z)$ and $\Lambda_{1}(z),\ldots,\Lambda_{r}(z)$ be a collection of non-zero polynomials. We say the collection $\{y_{1}(z),\ldots,y_{r}(z)\}$ is non-degenerate with respect to $\Lambda_{1}(z),\ldots,\Lambda_{r}(z)$ if the following conditions are satisfied:
    \begin{enumerate}[label=(\roman*)]
        \item $y_{i}(z)$ has no multiple zeros, $i=1,\ldots,r$;
     \item the zeros of $\Lambda_{i}(z)$ are different from the zeros of $y_{i}(z)$, $i=1,\ldots,r$;
    \item for $i,j=1,\ldots,r$ such that $i\neq j$ and $a_{ij}\neq 0$, the zeros of $y_{i}(z)$ and $y_{j}(z)$ are distinct. 
    \end{enumerate}
    When $\Lambda_{i}(z)$'s are clear from the context, we will simply say that $\{y_{1}(z),\ldots,y_{r}(z)\}$ is non-degenerate.
\end{definition}
We say a solution $\{q^{i}_{+}(z),q^{i}_{-}(z)\}_{i=1,\ldots,r}$ of the $qq$-system \eqref{qqsystem} is non-degenerate if it's positive part is non-degenerate, that is, $\{q^{i}_{+}(z)\}_{i=1,\ldots,r}$ is non-degenerate.

Let us recall the Bethe ansatz equations for the inhomogeneous Gaudin model. 
Denote the distinct roots of $\Lambda_{i}(z)$ by $z_{1}^{i},\ldots,z_{N_{i}}^{i}$, $i=1,\ldots,r$. Then using the multiplicities of the roots of $\Lambda_{i}(z)$'s we can define certain dominant integral coweights $\check{\mathrm{\lambda}}_{k}$ as:
    \[
    \Lambda_{i}(z)=\prod_{k=1}^{N_i}(z-z_{k}^{i})^{\langle\alpha_{i},\check{\mathrm{\lambda}}_{k}\rangle}
    \]
    Let $N:=\text{max}\{N_{i}:i=1,\ldots,r\}$ and let $V_{\check{\mathrm{\lambda}}_{1}},\ldots,V_{\check{\mathrm{\lambda}}_{N}}$ be the irreducible representations corresponding to the dominant integral coweights $\check{\mathrm{\lambda}}_{1},\ldots,\check{\mathrm{\lambda}}_{N}$.
\begin{definition}
    For each $i=1,\ldots,r$, fix positive integers $m_{1},\ldots,m_{r}$. The Bethe ansatz equations for the inhomogeneous Gaudin model corresponding to $Z^{H}\in\mathfrak{h}$, the representation $\otimes_{j=1}^{N}V_{\check{\mathrm{\lambda}}_{j}}$ and the coweight $\sum \check{\mathrm{\lambda}}_{j}-\sum m_{i}\check{\mathrm{\alpha}}_{i}$
    are the following equations in the unknowns $w_{l}^{i}$:
    \begin{equation}\label{betheqqg}
    \langle\alpha_{i},Z^H\rangle
    +
\sum_{j=1}^{N_{i}}\frac{\langle\alpha_{i},\check{\mathrm{\lambda}}_{j}\rangle}{w_{l}^{i}-z_{j}^{i}}
-
\sum_{(j,s)\neq(i,l)}\frac{a_{ji}}{w_{l}^{i}-w^{j}_{s}}=0,\quad i=1,\ldots,r, 
\quad
l=1,\ldots,m_{i}.
    \end{equation}
\end{definition}
\begin{remark}
    Any solution of the above Bethe ansatz equations should have that no $w_{l}^{i}$ is a root of $\Lambda_{i}(z)$ and if $(j,s)\neq(i,l)$ such that $a_{ji}\neq 0$, then $w_{l}^{i}\neq w^{j}_{s}$. 
\end{remark}
\begin{example}
    Let $\mathfrak{g}=\mathfrak{sl}_{2}$, $Z^{H}=\text{diag}(\zeta,-\zeta)$ and $\Lambda(z)=\prod_{k=1}^{N}(z-z_{k})^{n_{k}}$. Then the Bethe ansatz equations are
    \begin{equation}\label{bethesl2}
    2\zeta+\sum_{j=1}^{N}\frac{n_{j}}{w_{l}-z_{j}}
-
\sum_{s\neq l}\frac{2}{w_{l}-w_{s}}=0,\quad 
\quad
l=1,\ldots,m.
    \end{equation}
\end{example}
It is known (\cite{BSZ}*{Theorem 5.11}) that there is a surjective map from the set of non-degenerate polynomial solutions of the $qq$-system \eqref{qqsystem} to the set of solutions of Bethe ansatz equations, which takes a non-degenerate solution $\{q^{i}_{+}(z),q^{i}_{-}(z)\}_{i=1,\ldots,r}$ to $\{w_{l}^{i}\}_{i=1,\ldots,r,\text{ }l=1,\ldots,\text{deg}(q_{+}^{i})}$, where $w_{l}^{i}$ are the roots of $q_{+}^{i}(z)$.

For $\mathfrak{sl}_2$, as an application of Theorem \ref{sl2analyticity}, we can algorithmically solve the Bethe equations \eqref{bethesl2} for small enough parameter $t$. To state this result, we recall the following algorithm, which lifts a point of the tropical variety to a point of the variety up to any given order of $t$.
\begin{algorithm}
    (\cite{JMM}*{Algorithms 3.8 and 4.8}) Consider the polynomial ring $\mathbb{C}[t,x_{1},\ldots,x_{n}]$ and let $q_{i}\in\mathbb{C}[t,x_{1},\ldots,x_{n}]$, $1\leq i\leq s$. Let $J$ be the ideal generated by $q_{i}$'s in the ring $L[x_{1},\ldots,x_{n}]$.

    INPUT: $(m,\bm{w})\in\mathbb{N}_{>0}\times\mathbb{Q}^n$ such that $\bm{w}\in\trop(V(J))$.

    OUTPUT: $(N,p)\in\mathbb{N}\times\mathbb{C}[t,t^{-1}]^{n}$ such that $p(t^{1/N})$ coincides with the first $m$ terms of a solution 
    \hspace*{22mm}of $V(J)$ with $val(p)=\bm{w}$.
\end{algorithm}
\begin{remark}\label{software}
    The above algorithm is implemented in the software packages Gfan and SINGULAR in the case when $q_{i}\in \mathbb{Q}[t,x_{1},\ldots,x_{n}]$ for all $i$.
\end{remark}
We get the following result using Theorem \ref{tropicalqq}.
\begin{corollary}
   Assume that all the coefficients $d_{i}$'s of $\Lambda(z)$ are non-zero. Then for every given $m$, we can solve $\mathfrak{sl}_2$ Bethe Ansatz equations of the inhomogenous Gaudin model upto order $m$ of $t$ for small enough parameter $t$. If, moreover, $\Lambda(z)\in\mathbb{Q}[z]$, then there are software packages (Remark \ref{software}) to compute these solutions. 
\end{corollary}
\section{\texorpdfstring{$QQ$}{QQ}-systems}\label{proofsofQQ}
\subsection{Proofs}\label{proofQQ}
In this section, we give proofs of Theorems \ref{QQforgeneralg}, \ref{analyticityQQ}, \ref{SL2analyticity} and \ref{tropicalQQ}.

We first give the proof of Theorem \ref{SL2analyticity}. Then we will sketch the proofs of Theorems \ref{QQforgeneralg} and \ref{analyticityQQ}.
\begin{proof}(of Theorem \ref{SL2analyticity})
As in the proof of Theorem \ref{sl2analyticity}, let us write $\big(x_{1}(0),\ldots,x_{m}(0),y_{1}(0)\ldots,y_{n}(0)\big)$ as $(b_{1},\ldots,b_{m+n})$. Let us suppose that $l$ of the $b_i$'s are distinct, say they are $b_{j_{1}},\ldots,b_{j_{l}}$.

For $k=1,\ldots,m+n$, define
\[
\tilde{g}_{k}:=q^me_{k}\bigg(\frac{x_{1}}{q},\ldots,\frac{x_{m}}{q},x_{1},\ldots,x_{m}\bigg)
-
tq^n
e_{k}\bigg(\frac{x_{1}}{q},\ldots,\frac{x_{m}}{q},y_{1},\ldots,y_{n}\bigg)
-
d_{k}(q^{m}-tq^{n}),
\]
and consider the following map:
\[
\phi:\mathbb{C}^{m}\times\mathbb{C}^{n}\times\mathbb{C}\rightarrow\mathbb{C}^{m+n}
\]
\[
(x_{1},\ldots,x_{m},y_{1},\ldots,y_{n},t)\mapsto(\tilde{g}_{1},\ldots,\tilde{g}_{m+n}).
\]
Rest of the proof of Theorem \ref{SL2analyticity} is similar to proof of Theorem \ref{sl2analyticity}.
\end{proof}
\begin{proof}
    (of Theorem \ref{QQforgeneralg})
    The argument in this case is similar to the proof of Theorem \ref{qqforgeneralg}.
\end{proof}
\begin{proof}
    (of Theorem \ref{analyticityQQ}) The argument in this case is similar to the proof of Theorem \ref{analyticityqq}.
\end{proof}
\begin{proof}
(of Theorem \ref{tropicalQQ})
  The proof is similar to the proof of Theorem \ref{tropicalqq} and using the fact that 
  \[
  \nu\bigg(\frac{1}{q^{m}-tq^{n}}\bigg)
=1.
\]
\end{proof}

\subsection{XXZ spin chain}\label{applicationspinchain} In this section, we give an application of Theorem \ref{SL2analyticity} and Theorem \ref{tropicalQQ} to the Bethe ansatz equations of the XXZ model (\cite{FKSZ}) associated to $U_{q}\hat{\mathfrak{sl}}_{2}$.

Let us say that $u,v\in\mathbb{C}^{\times}$ are $q$-distinct if $q^{\mathbb{Z}}u\cap q^{\mathbb{Z}}v=\varnothing$.
\begin{definition}
    Let $y_{1}(z),\ldots, y_{r}(z)$ and $\Lambda_{1}(z),\ldots,\Lambda_{r}(z)$ be a collection of non-zero polynomials. We say the collection $\{y_{1}(z),\ldots,y_{r}(z)\}$ is $q$-nondegenerate with respect to $\Lambda_{1}(z),\ldots,\Lambda_{r}(z)$ if the following conditions are satisfied:
    \begin{enumerate}[label=(\roman*)]
     \item the zeros of $\Lambda_{i}(z)$ are $q$-distinct from the zeros of $y_{i}(z)$, $i=1,\ldots,r$;
    \item for $i,j=1,\ldots,r$ such that $i\neq j$ and $a_{ij}\neq 0$, the zeros of $y_{i}(z)$ and $y_{j}(z)$ are $q$-distinct. 
    \end{enumerate}
    When $\Lambda_{i}(z)$'s are clear from the context, we will simply say that $\{y_{1}(z),\ldots,y_{r}(z)\}$ is $q$-nondegenerate.
\end{definition}
We say a solution $\{Q^{i}_{+}(z),Q^{i}_{-}(z)\}_{i=1,\ldots,r}$ of the $QQ$-system \eqref{QQsystem} is $q$-nondegenerate if it's positive part is $q$-nondegenerate, that is, $\{Q^{i}_{+}(z)\}_{i=1,\ldots,r}$ is $q$-nondegenerate.

For the rest of this section, assume that $G$ is simply-laced. Let us recall the Bethe ansatz equations of the XXZ model associated to $U_{q}\hat{\mathfrak{g}}$. 
\begin{definition}
    The Bethe ansatz equations of the XXZ model associated to $U_{q}\hat{\mathfrak{g}}$ corresponding to $Z\in H$, the polynomials $\Lambda_{1}(z),\ldots,\Lambda_{r}(z)$ 
    are the following equations in the unknowns $w_{k}^{i}$:
    \[
    \frac{Q^{i}_{+}(qw_{k}^{i})}{Q^{i}_{+}(q^{-1}w_{k}^{i})}
    \prod_{j}\zeta_{j}^{a_{ji}}
    =
    -
    \frac{\Lambda_{i}(w^{i}_{k})\prod_{j>i}[Q^{j}_{+}(qw_{k}^{i})]^{-a_{ji}}\prod_{j<i}[Q^{j}_{+}(w_{k}^{i})]^{-a_{ji}}}{\Lambda_{i}(q^{-1}w^{i}_{k})\prod_{j>i}[Q^{j}_{+}(w_{k}^{i})]^{-a_{ji}}\prod_{j<i}[Q^{j}_{+}(q^{-1}w_{k}^{i})]^{-a_{ji}}}
    ,\quad i=1,\ldots,r, 
\quad
k=1,\ldots,m_{i}.
    \]
\end{definition}
\begin{example}
    Let $G=SL_{2}$, $Z=\text{diag}(\zeta,\zeta^{-1})$ and $\Lambda(z)=\prod_{p=1}^{L}\prod_{j_{p}=0}^{r_{p}-1}(z-q^{-j_{p}}z_{p})$. Then the Bethe ansatz equations are
    \begin{equation}\label{betheSL2}
    q^{r}\prod_{p=1}^{L}\frac{w_{k}-q^{1-r}z_{p}}{w_{k}-qz_{p}}
    =
    -\zeta^{2}q^{m}\prod_{j=1}^{m}\frac{qw_{k}-w_{j}}{w_{k}-qw_{j}}
    ,\quad 
\quad
k=1,\ldots,m,
    \end{equation}
    where $r=\sum_{p=1}^{L}r_{p}$.
\end{example}
It is known (\cite{FKSZ}*{Theorem 6.4}) that there is a one-to-one correspondence between the set of $q$-nondegenerate polynomial solutions of the $QQ$-system \eqref{QQsystem} and the set of solutions of the above Bethe ansatz equations, which takes a non-degenerate solution $\{Q^{i}_{+}(z),Q^{i}_{-}(z)\}_{i=1,\ldots,r}$ to $\{w_{k}^{i}\}_{i=1,\ldots,r,\text{ }k=1,\ldots,\text{deg}(Q_{+}^{i})}$, where $w_{k}^{i}$ are the roots of $Q_{+}^{i}(z)$.

For $SL_2$, as an application of Theorem \ref{SL2analyticity} and Theorem \ref{tropicalQQ}, we can algorithmically solve the Bethe equations \eqref{betheSL2}.
\begin{corollary}
    Assume that all the coefficients $d_{i}$'s of $\Lambda(z)$ are non-zero. Then for every given $m$, we can solve the Bethe ansatz equations of the XXZ model associated to $U_{q}\hat{\mathfrak{sl}}_{2}$ upto order $m$ of $t$ for small enough parameter $t$. If, moreover $\Lambda(z)\in\mathbb{Q}[z]$, then there are software packages (Remark \ref{software}) to compute these solutions. 
\end{corollary}

\bibliographystyle{alpha}
\bibliography{references}
\Addresses
\end{document}